\documentclass[11pt,reqno,a4paper]{amsart}

\usepackage{amsmath}
\usepackage{amsthm}
\usepackage{amssymb}

\usepackage{euscript,eufrak,verbatim,latexsym,amsmath,amssymb,graphics}

\usepackage[final]{epsfig}

\usepackage{graphicx}

\input amssym.def
\input  amssym.tex

\theoremstyle{plain}
\newtheorem{theorem}{Theorem}[section]
\newtheorem{proposition}[theorem]{Proposition}
\newtheorem{lemma}[theorem]{Lemma}

\newtheorem{remark}[theorem]{Remark}
\theoremstyle{plain}
  \newtheorem*{thm*}{Main Theorem}

\newtheorem{corollary}[theorem]{Corollary}

\def\a{\alpha}
\def\e{\epsilon}

\def\r{\rho}
\def\z{\zeta}
\def\b{\beta}
\def\g{\gamma}

\def\ud{\underline{d}_{\mu_\alpha}}
\def\nd{\underline{d}_{\nu}}

\def\R{\mathbb{R}}
\def\P{\mathbb{P}}
\def\N{\mathbb{N}}
\def\Z{\mathbb{Z}}

\def\P{\mathbb{P}}

\def\c{\mathcal C}

\DeclareMathOperator{\card}{card}
\DeclareMathSymbol{\varnothing}{\mathord}{AMSb}{"3F}

\title{The Hausdorff dimension of the set of dissipative points  for 
a Cantor-like model set for singly cusped parabolic dynamics}
\author{J\"org Schmeling and Bernd O. Stratmann}

\address{J\"org Schmeling:
     Department of Mathematics, LTH, University of Lund, 
     P.O. Box 118, SE-221 00 LUND, Sweden}
  \email{joerg@maths.lth.se}

  \address{Bernd O. Stratmann, Mathematical Institute, University of
    St. Andrews, North Haugh, St. Andrews KY16 9SS, Scotland}
  \email{bos@st-and.ac.uk}
  \keywords{Fractal geometry, Hausdorff dimension, Cauchy random walks, Kleinian groups}

  \subjclass[2000]{Primary: 60D05, 14L35, 51N30}

\begin{document}
\maketitle

\pagestyle{myheadings}

\markboth{The Hausdorff dimension of the set of dissipative points}{J. Schmeling 
and B. O. Stratmann}

\begin{abstract}
In this paper 
we introduce and study a certain intricate Cantor-like set $\c$ 
contained in unit interval. 
Our 
main result is to show that the set $\c$ itself, as well as
the set of dissipative points within $\c$,  both have  Hausdorff 
dimension equal to $1$.  The proof uses  the
transience of a certain non-symmetric Cauchy-type random walk.
\end{abstract}

\section{Introduction}\label{sec1}
\noindent In this paper we estimate the Hausdorff dimension of the set 
$\c_{\infty}$ of dissipative points within a certain Cantor-like 
subset $\c$ of the unit interval $[0,1) \subset \R$. There are two 
ways to define these sets. The first is via the interation of  a certain 
interval map $\Phi$, where $\c$ represents the set of points with an 
infinite forward orbit, and $\c_{\infty}$ is equal to the basin 
of attraction of the set of critical points of $\Phi$ (see Remark 
\ref{rem1} at the end of this introduction).  
The second construction is purely in terms of fractal geometry. Let 
us remark 
that our  motivation for considering the sets $\c$ and 
$\c_{\infty}$  stems from the investigations in \cite{PS} of the geometry of
limit sets of Kleinian groups with singly cusped parabolic dynamics. 
For the purposes of this paper this link is 
       irrelevant and therefore we omit the details. 
   However,  intuition coming from
   Kleinian groups has historically played a very important role in 
  in the development of Real and Complex Dynamics, and this paper can 
  be seen as 
  adding 
  to this tradition. 
 
\noindent Let us begin with by giving the slightly intricate, but more 
down-to-earth fractal geometric 
construction of 
the sets $\c$ and $\c_{\infty}$. For this we have to define certain 
families of fundamental intervals by induction as follows. We start with the unit 
interval $[0,1)$, and then partition the left 
half  of $[0,1)$  into the 
infinitely many intervals
\[  I_{1}:= 
\left[0,\frac{3}{\pi^{2}}\right), \, \, \hbox{and} \, \, 
I_{k+1}:= \left[ \frac{3}{\pi^{2}} \sum_{l=1}^{k} l^{-2},  
\frac{3}{\pi^{2}} \sum_{l=1}^{k+1} 
l^{-2}\right) ,\, \, \hbox{for} \, \, k \in \N.\]
The family of these first-level intervals will be denoted by 
$ C_{1}$.
Note that the right half $\left[\frac{1}{2}, 1\right)$ of the unit 
interval, which is clearly 
not captured by $C_{1}$,  should be 
interpreted as `the hole at the first level'.  
The second step is to partition each element $I_{k_{1}} \in C_{1}$ as 
follows. By starting from the left
  endpoint of  $I_{k_{1}}$, we partition the left half of 
  $I_{k_{1}}$
 into infinitely many mutually adjacent intervals 
$$I_{k_1 k_{1}+1}, \cdots,
  I_{k_1 k_1+l},\cdots,$$ where the diameters of these intervals are given by  
\begin{equation*}
\left|I_{k_1 k_1+l}\right|=\frac{3}{\pi^2}\frac{|I_{k_1}|}{l^2}, \, \, 
\hbox{for}\,\, l \in \N.
\end{equation*}
Similarly, by starting from the right
  endpoint  of  $I_{k_{1}}$ we insert into the right half 
  of $I_{k_{1}}$
  the $(k_{1}+1)$
  mutually adjacent intervals 
$$I_{k_1 k_{1}}, I_{k_1k_1-1}, \cdots,
   I_{k_1 0},$$ with diameters  given by
\begin{equation*}\label{length2}
\left|I_{k_1 k_1-l}\right|= 
\left\{
\begin{array}
{r@{\quad \hbox{for}\quad}l}     
\frac{3}{\pi^2}\frac{|I_{k_1}|}{(2k_{1})^2} &  l=0 \\
\frac{3}{\pi^2}\frac{|I_{k_1}|}{l^2} &   l \in \{1,\ldots,k_{1}\}.
\end{array} \right.\
\end{equation*}
The family of these second-level intervals will be denoted by 
$ C_{2}$.
Note that  in this way  we have perforated each  $I_{k_{1}}\in C_{1}$ 
such that there is a `hole'  in $I_{k_{1}}$  
with diameter  of 
order  $|I_{k_1}| / k_{1}$. \\
We then proceed by induction as follows.
Suppose that for $n \geq 2$  the $n$-th level interval  $I_{k_{1}\ldots k_{n}}$ has been 
constructed.  The 
$(n+1)$-th level intervals arising from $I_{k_{1}\ldots k_{n}}$
are then obtained as follows.
There are two 
cases to consider. The first case is that $k_{n}=0$, 
and here the 
partition only continues in the left half of  $I_{k_{1}\ldots k_{n-1}0}$.
More precisely, in this case we
start from  the left
  endpoint of  $I_{k_{1}\ldots k_{n-1}0}$ and  partition the 
  left half of $I_{k_{1}\ldots k_{n-1}0}$
 into infinitely many mutually adjacent intervals 
$$I_{k_1\cdots k_{n-1}0 1}, \cdots,
  I_{k_1\cdots k_{n-1}0 l},\cdots,$$  with diameters  given by  
\begin{equation*}\label{length1}
\left|I_{k_1\cdots 
k_{n-1}0 l}\right|=\frac{3}{\pi^2}\frac{|I_{k_1\cdots k_n}|}{l^2}, \, \, 
\hbox{for}\,\, l \in \N.
\end{equation*}
In the second case we have that $k_{n}\in \N$, and here we
start from the left
  endpoint of  $I_{k_{1}\ldots k_{n}}$ and partition the left half of 
  $I_{k_{1}\ldots k_{n}}$
 into infinitely many mutually adjacent intervals 
$$I_{k_1\cdots k_n k_n+1}, \cdots,
  I_{k_1\cdots k_nk_n+l},\cdots.$$ The diameters of these intervals
  are  
\begin{equation}\label{length1}
\left|I_{k_1\cdots 
k_nk_n+l}\right|=\frac{3}{\pi^2}\frac{|I_{k_1\cdots k_n}|}{l^2}, \, \, 
\hbox{for}\,\, l \in \N.
\end{equation}
Similarly, by starting from the right
  endpoint  of  $I_{k_{1}\ldots k_{n}}$ we insert into the right half of $I_{k_{1}\ldots k_{n}}$
  the $(k_{n}+1)$
  mutually adjacent intervals 
$$I_{k_1\cdots k_n k_n}, I_{k_1\cdots k_n k_n-1}, \cdots,
   I_{k_1\cdots k_n0},$$ with diameters  given by
\begin{equation}\label{length2}
\left|I_{k_1\cdots 
k_nk_n-l}\right|= 
\left\{
\begin{array}
{r@{\quad \hbox{for}\quad}l}     
\frac{3}{\pi^2}\frac{|I_{k_1\cdots k_n}|}{(2k_{n})^2} &  l=0 \\
\frac{3}{\pi^2}\frac{|I_{k_1\cdots k_n}|}{l^2} &   l \in \{1,\ldots,k_{n}\}.
\end{array} \right.\
\end{equation}
The so obtained set of intervals  of the $(n+1)$-th level will be 
denoted by $C_{n+1}$. 
That is,
$$ C_{n}= \{I_{k_{1}\ldots k_{n}}: k_{1}\in \N, k_{i+1} \in \N_{0} 
\, \hbox{ for } \,  
i \in \N , \hbox{ and if $k_{i} =0$ then $k_{i+1} \neq 0$}\}.$$
Again, note that by this we have perforated $I_{k_{1}\ldots k_{n}}$ 
such that in the first case `the hole' is precisely the right half of 
$I_{k_{1}\ldots k_{n}}$, whereas in the second case the diameter of the hole is of order $|I_{k_{1}\ldots 
k_{n}}|/k_{n}$. 
Also, let us emphasize that by construction, the state $0$  necessarily  has to renew 
itself. That is, 
the generation following the interval $I_{k_1\cdots 
k_{n-1}0}$ is given by  $\{I_{k_1\cdots 
k_{n-1}0k_{n+1}}: k_{n+1} \in \N\}$.  Moreover,
note that the system can only be stationary at 
states $k_{n} \in \N$, which means that if $I_{k_{1}\ldots k_{n}}$ is a 
given interval of some level $n$ then $I_{k_{1}\ldots 
k_{n}k_{n}}$ exists if and only if $k_{n}\ne 0$.
Finally, note that we 
always assume that the intervals $I_{k_{1} \ldots k_{n}}$
are half open, namely closed to the left and open to the right.
  
\noindent With this inductive construction of the generating 
intervals $I_{k_{1}\ldots k_{n}}$ at hand, the Cantor-like set $\c$  
is defined by   
\[
\c:=\bigcap_{n\in\N} \bigcup_{I\in C_{n}} I.
\]
Next, we define the set of  dissipative points in $\c$.
For this we require the following canonical coding of the elements in
$\c$.  A 
finite or infinite sequence $(k_{1}, k_{2},\ldots)$ is called admissable 
if   $I_{k_1\cdots k_n} \in C_{n}$, for all $n \in \N$.
Clearly, the diameter of $I_{k_1\cdots k_n}$ tends  to zero 
as $n$ tends to $\infty$, for every fixed infinite 
admissable sequence $(k_n)_{n\in \N}$, and therefore, 
\[
\bigcap_{n=1}^\infty I_{k_1\cdots k_n} \, \, \hbox{is a singleton}.
\]
In particular, each $x\in \c$  is coded  uniquely by an infinite 
admissable sequence, and this gives rise to
the bijection
\[  \rho: \Sigma \to \c,  (k_{1},k_{2},\ldots) \mapsto  
\bigcap_{n=1}^\infty I_{k_1\cdots k_n},\]
where $\Sigma$ refers to the set of all admissable 
sequences.
Using this coding,  the set $\c_{\infty}\subset \c$ of dissipative points 
 is then given by
\[
\c_\infty :=\left\{x\in \c \, :\, x = \rho(k_{1},k_{2},\ldots) \, \, 
\hbox{and} \, \, \lim_{n\to\infty} k_{n} =\infty\right\}.
\]
The following theorem gives the main result of this paper. Here, 
$\dim_{H}$ refers to the Hausdorff dimension.

\begin{thm*}\label{main}
    For $\c$ and $\c_{\infty}$ as defined above, we have
    \[
\dim_H \left(\c_\infty\right) = \dim_{H} (\c)=1.
\]
\end{thm*}

\vspace{2mm}

\begin{remark}\label{rem1}   {\rm
    As already mentioned at the beginning of the introduction,  the sets $\c$ and 
    $\c_{\infty}$ 
    can alternatively be defined in terms of a certain interval map $$\Phi:
   \bigcup_{I_{k} \in 
    C_{1}} I_{k} \to [0,1).$$ 
Namely, the map $\Phi$ is given piecewise by    $\Phi|_{I_{k}}:=\phi_{k}$, 
for each $k \in \N$, where the maps $\phi_{k}:\bigcup_{I_{kl} \in 
C_{2}}I_{kl} \to I_{k}$  
are piecewise linear in the following sense. For each $k\in \N$ and 
$l \in \N_{0}$, the map  $\phi_{k}|_{I_{kl}}$ is a linear 
and
    \[ \phi_{k} |_{_{I_{kl}}} \left(I_{kl}  \right) =  \left\{
\begin{array}
{r@{\quad }l}     
\bigcup_{m=l-1}^{\infty} I_{km}&  \hbox{for}  \quad \,  l\neq 0 \\
  & \\
I_{k} \setminus I_{k0} &  \hbox{for}   \quad \, l = 0 .
\end{array} \right.\
\]
One then immediately verifies that the set $\c$ is  equal to the set of  
points which have  
an infinite forward orbit under $\Phi$.  Moreover, the centre $c_{k}$ of 
$I_{k}$ is the critical point of $\phi_{k}$, for each $k \in \N$, and 
thus 
$Crit(\Phi):= \{c_{k}:k\in \N\}$ represents the countable set of critical points 
of $\Phi$.  
With $\omega_{\Phi} (x)$ referring to the $\omega$-limit set of an 
element $x \in [0,1/2)$ with respect to $\Phi$ (that is, $\omega_{\Phi}(x)$ 
denotes the set of accumulation points of $\{\Phi^{n}(x): n \in \N\}$), 
and with
$\mathcal{B}(Crit(\Phi)) :=\{ x: 
\omega_{\Phi} (x) \subset Crit(\Phi) \}$ denoting  the basin of attraction of 
$Crit(\Phi)$ under 
$\Phi$, we then have
\[ \c_{\infty} =  \mathcal{B}(Crit(\Phi)).\]
 Similar types of interval maps have been studied for 
 instance in  \cite{BKNS}  and \cite{SV} in connection with `wild 
 Cantor sets' and the search for Julia sets of positive Lebesgue 
 measure.  It seems worthwhile to point out that the `Martingale Argument' of Keller (see 
 \cite{BKNS}, Section 4.1), which gives a criterion for the basin of 
 attraction of a critical point to be of positive Lebesgue measure,  is not applicable to
 the map $\Phi$ and hence does not allow to draw any conclusion 
 for the Lebesgue measure of $\c_{\infty}$. Nevertheless, 
 recent studies in the theory of Kleinian groups
 (cf. \cite{A} \cite{BCM} \cite{CG}, and also 
 \cite{C}) 
 have confirmed
 the Ahlfors Conjecture, and applying these  to 
 our situation here strongly  suggests that  $\c$ and $\c_{\infty}$ 
 are both of $1$-dimensional 
 Lebesgue measure equal to $0$. However, currently it is still a conjecture
 that the Lebesgue measure 
 of $\c$ and $\c_{\infty}$
 is equal to zero, and  it would be desirable to 
 have an 
 elementary proof of this conjecture.}
    \end{remark}

  \noindent  {\bf Acknowledgement:}
    We like to thank the {\em EPSRC}  for supporting a one 
    week visit of the first author to 
    the School of Mathematics at the University of
    St. Andrews.  During this visit we began with the work towards 
    this paper. Also, we like to thank the {\em Erwin Schršdinger 
    Institute, 
    Vienna} for supporting the workshop {\sl Ergodic Theory - Limit Theorems 
    and Dimensions} (Vienna, 17 - 21 December 2007)
    during which we made the finishing touches to this paper. Finally,
    we are grateful to
    Tatyana Turova for very helpful discussions concerning the random 
    walk argument which we employ here.

 \section{{\bf Proof of  the Main Theorem}}\label{proof}
\noindent Since $\c_\infty\subset \c \subset [0,1)$, we have
\[
\dim_H \left(\c_\infty\right)\le \dim_{H}\left(\c\right) \le 1.
\]
Therefore, our strategy will be to construct a family of probability 
measures $\mu_\a$  on $\c_\infty$ , for $1/2<\a<1$, such that the 
Hausdorff dimension  
$\dim_H(\mu_\a)$ of the measure $\mu_{\a}$ tends to $1$ for $\a$ tending to $1$. Clearly, this
will then 
be sufficient for the proof of the Main Theorem.

\subsection{The family of measures $\mu_{\a}$}

$\,$

\noindent Let $1/2< \a \le 1$ be fixed. We then define a set function 
$\mu_{\a}$ on the
intervals $I_{k_1\cdots k_n}$ by induction in the following way. 
Define $I_{0}:=[0,1)$ and set $I_{0k}:=I_{k}$ for all $k \in \N$.  
Then let
\[
\mu_\a(I_{0}) :=1, 
\]
and define $\mu_{\a}\left(I_{k_1\cdots 
k_{n}k_{n+1}}\right)$ for each finite admissable sequence $(k_1,\cdots 
,k_{n+1})$  as follows. With $\z(s):=\sum_{m=1}^{\infty} m^{-s}$ referring to the Riemann zeta 
function,  we define for $k_{n}\neq k_{n+1}$,
\[ 
\mu_{\a}\left(I_{k_1\cdots k_{n} 
k_{n+1}}\right):=  \left\{
\begin{array}
 {r@{\, \hbox{for}\,}l}   \frac{\mu_\a\left(I_{k_1\cdots k_n}\right)}{2\z 
(2\a)}\left(\frac{1}{|k_{n+1}-k_n|^{2\a}}+\frac{1}{(k_{n+1}+k_n)^{2\a}}\right)     &  k_{n+1}\ne 0 \\
\frac{\mu_\a\left(I_{k_1\cdots k_n}\right)}{2\z 
(2\a) }  \frac{1}{k_n^{2\a}}&   k_{n+1}=0.
\end{array} \right.\] 
Also,  for $k_{n}=k_{n+1}$ let
\[\mu_{\a}\left(I_{k_1\cdots 
k_nk_{n}}\right):= \frac{\mu_\a\left(I_{k_1\cdots k_n}\right)}{2\z 
(2\a) }  \frac{1}{(2k_n)^{2\a}}. \]
On the first sight, this definition of the set function $\mu_{\a}$ might appear to
be slightly artificial. However, in the next 
section we will see that this definition reflects  the transition probabilities 
of a certain (transient) random walk on $\N_{0}$, and therefore is 
rather canonical. Before we come to this, 
let us first state the following consistency property for $\mu_{\a}$. 
This property can also be deduced using the random walk 
of  Section \ref{random walk}. Nevertheless, the following gives an 
elementary proof of this 
consistency property.
\begin{lemma}\label{consistency}
For each finite admissable sequence $(k_1,\cdots ,k_n)$, we have
\[
\mu_\a\left(I_{k_1\cdots k_n}\right)=\sum_{k_{n+1}\ge 0}\mu_{\a}\left(I_{k_1\cdots k_nk_{n+1}}\right).
\]
\end{lemma}
\begin{proof}
For $k_{n}=0$, we have
\[ \sum_{k_{n+1}\ge 0 \atop k_{n+1}\ne 
0}\mu_{\a}\left(I_{k_1\cdots k_{n-1} 0 k_{n+1}}\right)
= \frac{\mu_\a\left(I_{k_1\cdots k_{n-1} 0}\right)}{2\z 
(2\a)} \sum_{l=1}^{\infty} \frac{2}{l^{2\a}} = \mu_\a\left(I_{k_1\cdots 
k_{n-1} 0}\right).\]
If $k_{n}\neq 0$, then we compute
\begin{align*}
&\sum_{k_{n+1}\ge 0}\mu_{\a}\left(I_{k_1\cdots 
k_nk_{n+1}}\right)\\&
=\sum_{k_{n+1}>0 \atop k_{n+1}\ne k_n}\frac{1}{2\z 
(2\a)}\left(\frac{1}{|k_{n+1}-k_n|^{2\a}}+\frac{1}{(k_{n+1}
+k_n)^{2\a}}\right)
\mu_\a\left(I_{k_1\cdots 
k_n}\right)+\\&
\phantom{======}+\frac{1}{2\z 
(2\a) k_n^{2\a}}\mu_\a\left(I_{k_1\cdots k_n}\right) +
\frac{1}{2\z 
(2\a) (2k_n)^{2\a}}\mu_\a\left(I_{k_1\cdots k_n}\right)\\&
=\frac{\mu_\a\left(I_{k_1\cdots k_n}\right)}{2\z 
(2\a)}\left(\sum_{k_{n+1}>0 \atop k_{n+1} \ne 
k_{n}}\Big(\frac{1}{|k_{n+1}-k_n|^{2\a}}+\frac{1}{(k_{n+1}
+k_n)^{2\a}}\right)
+\\&
\phantom{======}+\frac{1}{k_n^{2\a}}
+\frac{1}{(2k_n)^{2\a}}\Big)\\&
=\frac{\mu_\a\left(I_{k_1\cdots k_n}\right)}{2\z 
(2\a)}\Big(\sum_{k_{n+1}>k_n}
\frac{1}{(k_{n+1}-k_n)^{2\a}}+\sum_{0<k_{n+1}<k_n}
\frac{1}{(k_{n}-k_{n+1})^{2\a}}+\\&
\phantom{==}+\sum_{0<k_{n+1}<k_{n}}\frac{1}{(k_{n+1}+k_n)^{2\a}}
+\sum_{k_{n+1}>k_{n}}\frac{1}{(k_{n+1}+k_n)^{2\a}}+\frac{1}{k_n^{2\a}} 
+\frac{1}{(2k_n)^{2\a}}\Big)\\&
=\frac{\mu_\a\left(I_{k_1\cdots k_n}\right)}{2\z (2\a)}\cdot 
2\sum_{l=1}^\infty\frac{1}{l^{2\a}}=\mu_\a\left(I_{k_1\cdots 
k_n}\right).
\end{align*} \end{proof}
\noindent The 
following  is an immediate consequence of the previous lemma.

\begin{corollary}\label{C}
 The measure  $\mu_{\a}$ 
    is a probability measure on $\c$.    
    \end{corollary}

\subsection{The associated random walk}\label{random walk} 

$\,$

\noindent In this section we show that the measure $\mu_{\a}$ can be 
interpreted in terms of a certain random 
walk. In particular, this will give 
that $\mu_{\a}$ has the Markov property. For this, let the random variables 
$\tilde{X}^\a_n$ be defined by the probability (with respect to 
$\mu_\a$)  being in 
the interval $I_{k_1\cdots k_nk_{n+1}}$ given that in the previous step 
the process has been in the interval $I_{k_1\cdots
  k_n}$.  That is, the random variables 
$\tilde{X}^\a_n$ is given as follows.
\begin{itemize}
    \item
  For $k_{n+1}>0$ such that $k_{n+1} \neq k_{n}$, let
\begin{align*}
\P&\left(\tilde{X}^\a_{n+1}=I_{k_1\cdots k_nk_{n+1}}\, |\, 
\tilde{X}^\a_n=I_{k_1\cdots
    k_n}\right):=\frac{\mu_{\a}\left(I_{k_1\cdots
        k_nk_{n+1}}\right)}{\mu_{\a}\left(I_{k_1\cdots
        k_n}\right)}\\&
=\frac{1}{2\z 
(2\a)}\left(\frac{1}{|k_{n+1}-k_n|^{2\a}}+\frac{1}{(k_{n+1}+k_n)^{2\a}}\right).
\end{align*}
\item For $k_{n+1}>0$ such that $k_{n+1} = k_{n}$, let
\[
\P \left(\tilde{X}^\a_{n+1}=I_{k_1\cdots k_nk_{n}}\, |\, 
\tilde{X}^\a_n=I_{k_1\cdots
    k_n}\right):=\frac{\mu_{\a}\left(I_{k_1\cdots
	k_nk_{n}}\right)}{\mu_{\a}\left(I_{k_1\cdots
	k_n}\right)}
=\frac{1}{2\z (2\a)}\frac{1}{(2k_n)^{2\a}}.
\]
\item If $k_{n+1}=0$, then
\[
\P \left(\tilde{X}^\a_{n+1}=I_{k_1\cdots k_n0}\, |\, 
\tilde{X}^\a_n=I_{k_1\cdots
    k_n}\right):=\frac{\mu_{\a}\left(I_{k_1\cdots
        k_n0}\right)}{\mu_{\a}\left(I_{k_1\cdots
        k_n}\right)}
=\frac{1}{2\z (2\a)}\frac{1}{k_n^{2\a}}.
\]
\end{itemize}
Clearly, these conditional probabilities do not depend on 
  $k_1,\cdots ,k_{n-1}$. Hence, we can define an associated random walk 
  $X_n^\a$ on $\N_{0}$ by 
the following transition probabilities.
\begin{itemize}
    \item
For $l,m \in \N_{0}$, let
\[ 
\P (X^\a_{n+1}=l\, |\, X^\a_n=m):=  \left\{
\begin{array}
 {r@{\quad \hbox{for}\quad}l}     
 \frac{1}{2\z 
 (2\a)}\left(\frac{1}{|m-l|^{2\a}}+\frac{1}{(m+l)^{2\a}}\right) & l \ne 
 0,  l\ne 
 m\\
 \frac{1}{2\z (2\a)}\frac{1}{(2m)^{2\a}} &   l \neq 0, l=m\\
 \frac{1}{2\z (2\a)}\frac{1}{m^{2\a}}&   l=0,m\neq 0\\
 0 & l=m=0.
 \end{array} \right.\] 
 \end{itemize}
 The random walk $X_n^\a$ is very closely connected to our original 
geometric setting, since it allows to
recover the measure $\mu_\a$ as follows.
\begin{align*}
&\P(X^\a_1=k_1, \cdots , X^\a_n=k_n)\\&=\P(X^\a_n=k_n\, |\, 
X^\a_{n-1}=k_{n-1})\cdot\P(X^\a_{n-1}=k_{n-1}\, |\, 
X^\a_{n-2}=k_{n-2}) \cdot \\&\phantom{====} . \,  . \,  . \, \cdot\P(X^\a_1=k_1\, 
|\, X^\a_0=0)\\&
=\frac{\mu_{\a}\left(I_{k_1\cdots
        k_n}\right)}{\mu_{\a}\left(I_{k_1\cdots
        k_{n-1}}\right)}\frac{\mu_{\a}\left(I_{k_1\cdots
        k_{n-1}}\right)}{\mu_{\a}\left(I_{k_1\cdots
        k_{n-2}}\right)}\cdots\frac{\mu_{\a}\left(I_{k_1}\right)}{\mu_{\a}\left([0,1)\right)} 
=\mu_{\a}\left(I_{k_1\cdots
        k_n}\right).
\end{align*}
The aim now is to show that the random walk $X_n^\a$ is transient. This will 
then allow us to deduce that 
 $\mu_\a$ is non-trivial on $\c_\infty$.
\begin{theorem}\label{RW}
For each  $1/2<\a<1$, the random walk $X_n^{\a}$ on $\N_{0}$ is transient.  That is, we  have 
$\P$-almost 
surely,
\[
\lim_{n\to\infty}X_n^{\a}= \infty.
\]
\end{theorem}
\begin{proof}
Let $1< \b <2$ be fixed, and consider the Cauchy-type random walk $Y^\b_n$ on 
$\Z$, given by the transition probabilities
\[
\P(Y^{\b}_{n+1}=m+l\, |\, Y^{\b}_n=m):=\frac{1}{2\z (\b)} \frac{1}{|l|^{\b}} \, \, \hbox{for} \, \,  n \in \N, m \in \Z \, 
\, \hbox{and} \, \, l\in\Z\setminus\{0\}.
\]
It is well known that $Y^\b_n$ is symmetric, and that $Y^\b_n$ is transient 
if and only if 
$\b<2$.  Let $\tilde{Y}^\b_n$ 
denote the random walk which arises from $Y^\b_n$ in the following 
way. Let $l_1, \cdots ,l_n$ be the sequence of jumps of $Y^\b_n$
after the first $n$ steps. 
With $r_1$ referring to the first time at which $Y^\b_k$ crosses $0$, we set 
$\tilde{Y}^\b_{r_1}:=|Y^\b_{r_1}|$. Subsequently,  we apply the jumps 
$|l_{r_1+1}|$ up to $|l_{r_2}|$, where $r_2$ refers to the next time $Y^\b_n$ 
crosses $0$. After that, we apply the jumps $|l_{r_2+1}|$ up to 
$|l_{r_3}|$, where  
$r_3$ denotes the next  time at which the process $Y^\b_k$ first crosses 
$0$ again. 
More precisely, with $r_i$ referring to 
the $i$-th time the process $Y^\b_k$ crosses $0$, we let 
$\tilde{Y}^\b_{r_i}:=|Y^\b_{r_i}|$, and between each  two consecutive  
crossings $r_{i}$ and $r_{i+1}$
 we define the jumps of $\tilde{Y}^\b_{r_i}$ to be
$|l_{r_i+1}|, \ldots, |l_{r_{i+1}}|$.  Note that the random 
walk we have just described is equal to the random walk $|Y^\b_n|$.
Also, note that since 
$Y^\b_n$ is a symmetric random walk,  the above modification of the sample 
path does not alter its probability. 
Therefore,
it immediately follows that the transience of
$Y^\b_n$ implies that $\tilde{Y}^\b_n$  is 
 transient. Using the fact that by symmetry of  $Y^\b_n$ we have 
$\P(Y^\b_n=m)=\P(Y^\b_n=-m)$,
we now  compute the transition probabilities of the random walk 
$\tilde{Y}^\b_n$ on $\N_{0}$ as follows. 
For 
$l, m\in \N_{0}$ such that $l \ne 
 0,$ and $ m\ne 
 l$, we have 
\begin{align*}
&\P(\tilde{Y}^\b_{n+1}=l\, |\, 
\tilde{Y}^\b_n=m)=\P(|Y^\b_{n+1}|=l\, |\, |Y^\b_n|=m)\\&
=\frac{\P(|Y^\b_{n+1}|=l,\, Y^\b_n=m\text{ or } 
Y^\b_n=-m)}{\P(Y^\b_n=m\text{ or } Y^\b_n=-m)}\\&
=\frac{\P(|Y^\b_{n+1}|=l,\, 
Y^\b_n=m)}{2\P(Y^\b_n=m)}+\frac{\P(|Y^\b_{n+1}|=l,\, 
Y^\b_n=-m)}{2\P(Y^\b_n=-m)}\\&
=\frac{\P(|Y^\b_{n+1}|=l , \, Y^\b_n=m)}{\P(Y^\b_n=m)}
=\frac{1}{2\z 
(\b)}\left(\frac{1}{|m-l|^{\b}}+\frac{1}{(m+l)^{\b}}\right).
\end{align*}
Similarly, we obtain for $l=m \neq 0$,
\[
\P(\tilde{Y}^\b_{n+1}=m \, |\, 
\tilde{Y}^\b_n=m)=\frac{1}{2\z 
(\b)}\frac{1}{(2m)^{\b}},
\]
and for $l =0$ and  $m \neq l$,
\[
\P(\tilde{Y}^\b_{n+1}=0 \, |\, 
\tilde{Y}^\b_n=m)=\frac{1}{2\z 
(\b)}\frac{1}{m^{\b}}.
\]
Finally, note that 
we immediately have  
\[\P(\tilde{Y}^\b_{n+1}=0\, |\, 
\tilde{Y}^\b_n=0) =0.\]
This shows that the  transition probabilities of  $\tilde{Y}^\b_n$ coincide 
with the ones of
$X^{\beta/2}_n$. Therefore, since  $\tilde{Y}^\b_n$ is transient, 
it follows that $X_n^{\beta/2}$ is transient.
This finishes the proof of the theorem.
\end{proof}

\noindent As already announced before,  Theorem~\ref{RW}  has the following
important implication.
\begin{corollary}\label{transience}
For every $1/2<\a<1$, we have  
\[
\mu_\a(\c_\infty)=1.
\]
\end{corollary}

\begin{remark}
    {\rm
Note that the proof of Theorem~\ref{RW}  relies heavily on the 
fact that $1/2< \a <1$. 
Namely, for instance for $\a=1$ the associated random walk is recurrent, 
and consequently
the 
measure $\mu_{\a}$ vanishes on $\c_\infty$.}
\end{remark}

\subsection{Approximating  the essential support of $\mu_{\a}$}

$\,$

\noindent In order to prepare our estimate of the lower pointwise dimension
of $\mu_{\a}$, we need a further approximation of the 
essential 
support of this measure. We will see that $\mu_{\a}$-almost 
surely the diameters of the coding intervals of an element of $\c_{\infty}$ do 
not shrink too fast. For this  we define, for $\g \in \R$,
\[
\c_\infty^{\g}:=\left\{x\in \c_\infty\, :\, x=\r (k_1,k_{2}, \cdots) \text{ 
such that 
}\limsup_{n\to\infty}\frac{|k_{n+1}-k_n|}{n^\g}\le 1\right\}.
\]
\begin{lemma}\label{slow}
For each $1/2<\a <1$ and $\g>1/(2\a-1)$, we have
\[
\mu_\a(\c_\infty^{\g})=1.
\]
\end{lemma}
\begin{proof}
The proof is an easy consequence of the Borel-Cantelli lemma. Indeed, first
note that for $\b=2\a$ and $k \in \N$ we have
\[
\P(|Y^{\b}_{n}- Y^{\b}_{n-1}| \ge k) \ge 
 \P(|\tilde{Y}^{\b}_{n}-\tilde{Y}^{\b}_{n-1}|\ge k).
\]
The latter is an immediate consequence of the fact that the random walk  $Y^\b_n$ has the same 
distribution as $\tilde{Y}^\b_n$, but without
reflections at $0$. Hence, it suffices to prove the lemma for the
symmetric random walk $Y^\b_n$.
For this we define
\[
p_n^\g:=\P(|Y^\b_n-Y^\b_{n-1}|\ge n^\g).
\]
We then have, for each $n \in \N$ and with $c(\b)>0$ referring to
some universal constant, 
\[
p_n^\g=2\sum_{k=n^\g}^\infty\frac{1}{k^\b}\le 
c(\b)\frac{1}{n^{\g(\b-1)}}.
\]
Since the series $\sum_{n=1}^\infty 
p_n^\g$
converges for $\g>1/(\b-1)$,  the Borel-Cantelli Lemma implies that $\P$-almost 
surely there are at most finitely many $n$ which satisfy the inequality
\[
|Y^\b_n-Y^\b_{n-1}|\ge n^\g.
\]
This shows that for $\mu_{\a}$-almost every 
 $x = \r(k_{1},k_{2},\ldots) \in \c$ we have
\[
\limsup_{n\to\infty}\frac{|k_{n+1}-k_n|}{n^\g}\le 1. 
\]
\end{proof}

\subsection{The lower pointwise dimension on fundamental intervals}

$\,$

\noindent The main result of this section will be the following estimate 
for the lower 
pointwise dimension  of the measure $\mu_\a$ restricted to the fundamental intervals 
$I_{k_{1}\ldots k_{n}}$. 

\begin{proposition}\label{lpd}
    For each $\e >0$ there exists $1/2<\a <1$ and $\g>1/(2\a-1)$ such 
    that for 
   every $x= \r(k_{1},k_{2}, \ldots ) \in \c_{\infty}^{\g}$,    
 \[  \liminf_{n \to \infty} \frac{\log\mu_\a(I_{k_1\cdots k_n})}{\log|I_{k_1\cdots
    k_n}|}  \geq \a -  \e.\]
 Furthermore, in here we have that  $\a$ tends to $1$ for $\e$ 
tending to $0$.
    
    \end{proposition}

 \begin{proof}
Since we are interested in the asymptotic behaviour of  $I_{k_{1}\ldots 
k_{n}}$ for points
$x = \r (k_{1},k_{2},\ldots) \in \c_\infty^\g$, Corollary~\ref{transience} 
and Lemma~\ref{slow} imply that we can assume without loss of generality 
that $k_n>0$
and $|k_{n+1}-k_n|\le n^\g$, for all $n \in \N$. Furthermore, for ease 
of exposition we only consider sequences which do not contain 
repetitions. That is, we assume that $k_{n}\neq k_{n+1}$, for all 
$n\in \N$. The case with repetitions  
can be dealt  with in similar way and is left to the reader.
Using the definition of $\mu_{\a}$,  we then have
\begin{align*}
&\frac{\log\mu_\a(I_{k_1\cdots k_nk_{n+1}})}{\log|I_{k_1\cdots
    k_nk_{n+1}}|}\\&=\frac{-\log(2\z (2\a))}{\log|I_{k_1\cdots
    k_nk_{n+1}}|}+\frac{\log\left[\left(\frac{1}{|k_{n+1}-k_n|^{2\a}}+
    \frac{1}{(k_n+k_{n+1})^{2\a}}\right)\mu_\a(I_{k_1\cdots
    k_n})\right]}{\log|I_{k_1\cdots
    k_nk_{n+1}}|}\\&
\ge\frac{-\log(2\z (2\a))}{\log|I_{k_1\cdots
    k_nk_{n+1}}|}+\frac{\log\left[\frac{2}{|k_{n+1}-k_n|^{2\a}}\mu_\a(I_{k_1\cdots
    k_n})\right]}{\log|I_{k_1\cdots
    k_nk_{n+1}}|}\\&
=\frac{-\log(2\z (2\a))}{\log|I_{k_1\cdots
    k_nk_{n+1}}|}+\frac{\log\left[\frac{2\cdot 2^\a\z 
(2)^\a|I_{k_1\cdots
          k_{n+1}}|^\a}{|I_{k_{1}\cdots 
k_n}|^{2\a}}\mu_\a(I_{k_1\cdots
    k_n})\right]}{\log|I_{k_1\cdots
    k_nk_{n+1}}|}\\&
=\frac{\log\frac{2^\a\z (2)^\a}{\z (2\a)}}{\log|I_{k_1\cdots
    k_nk_{n+1}}|}+\frac{\log(|I_{k_1\cdots 
k_{n+1}}|^\a)}{\log(I_{k_1\cdots 
k_{n+1}})}+\frac{\log\frac{\mu_\a(I_{k_1\cdots
    k_n})}{|I_{k_1\cdots k_n}|^\a}}{\log|I_{k_1\cdots
    k_nk_{n+1}|}}\\&
= \a+ \frac{\log\frac{2^\a\z (2)^\a}{\z (2\a)}}{\log|I_{k_1\cdots
    k_nk_{n+1}}|}+\frac{\log\frac{\mu_\a(I_{k_1\cdots
    k_n})}{|I_{k_1\cdots k_n}|^\a}}{\log|I_{k_1\cdots
    k_nk_{n+1}}|}.
\end{align*}
Since 
\begin{equation}\label{1}
|I_{k_1\cdots k_{n+1}}|<\left(\frac{1}{2\z (2)}\right)^n,
\end{equation}
it follows for each $\kappa>0$ and for all $n$ sufficiently 
large,
\begin{equation}\label{2nd}
 \frac{\log\frac{2^\a\z (2)^\a}{\z 
(2\a)}}{\log|I_{k_1\cdots
    k_nk_{n+1}}|}> - \kappa.
\end{equation}
Clearly, we even have that the limit of the latter expression is equal 
to $0$.
This settles the second term in the final line in the above calculation. The third term is
more subtle, and for this we proceed as follows. 
Using
\eqref{length1} and~\eqref{length2}, we derive with 
the convention $k_{0}\equiv 0$,
\[
|I_{k_1\cdots k_n}|^\a=\prod_{i=0}^{n-1}\left(\frac{1}{2^\a\z 
(2)^\a}\frac{1}{|k_{i+1}-k_i|^{2\a}}\right).
\]
Similarly, using the recursive definition of $\mu_{\a}$, we obtain
\[
\mu_\a(I_{k_1\cdots k_n})=\prod_{i=0}^{n-1}\left(\frac{1}{2\z 
(2\a)}\left[\frac{1}{|k_{i+1}-k_i|^{2\a}}+\frac{1}{(k_{i+1}+k_i)^{2\a}}\right]\right).
\]
Hence,
\begin{align*}
&\frac{\log\frac{\mu_\a(I_{k_1\cdots
    k_n})}{|I_{k_1\cdots k_n}|^\a}}{\log|I_{k_1\cdots
    k_nk_{n+1}}|}=\frac{\log\frac{\prod_{i=0}^{n-1}\left(\frac{1}{2\z
        (2\a)}\left[\frac{1}{|k_{i+1}-k_i|^{2\a}}+\frac{1}{(k_{i+1}+k_i)^{2\a}}\right]\right)}{\prod_{i=0}^{n-1}\left(\frac{1}{2^\a\z 
(2)^\a}\frac{1}{|k_{i+1}-k_i|^{2\a}}\right)}}{\log\left(\prod_{i=0}^{n}\left(\frac{1}{2^\a\z 
(2)^\a}\frac{1}{|k_{i+1}-k_i|^{2\a}}\right)\right)}\\&
=\frac{\log\left(\left[\frac{2^\a\z(2)^\a}{2\z(2\a)}\right]^n\prod_{i=0}^{n-1}\left[1+\left(\frac{|k_{i+1}-k_i|}{k_{i+1}+k_i}\right)^{2\a}\right]\right)}{\log\left(\left[\frac{1}{2^\a\z
        (2)^\a}\right]^{n+1}\prod_{i=0}^{n}\frac{1}{|k_{i+1}-k_i|^{2\a}}\right)}\\&
=\frac{n\a\log(2\z(2))-n\log(2\z(2\a))+\sum_{i=0}^{n-1}\log\left(1+\left[\frac{|k_{i+1}-k_i|}{k_{i+1}+k_i}
\right]^{2\a}\right)}{-(n+1)\log(2\z(2))+\sum_{i=0}^n\log\frac{1}{|k_{i+1}-k_i|^{2\a}}}\\&
=\frac{\log(2\z(2\a)) -\a\log(2\z(2))-\frac{1}{n}\sum_{i=0}^{n-1}\log\left(1
+\left[\frac{|k_{i+1}-k_i|}{k_{i+1}+k_i}\right]^{2\a}\right)}{\frac{n+1}{n}\log(2\z(2))
+\frac{1}{n}\sum_{i=0}^n\log |k_{i+1}-k_i|^{2\a}}.
\end{align*}
Let $\kappa >0$ be fixed. We then distinguish the following two cases. 
\\
First, if for some $n\in \N$ we have
\[
\frac{1}{n}\sum_{i=0}^{n-1}\log\left(1+
\left[\frac{|k_{i+1}-k_i|}{k_{i+1}+k_i}\right]^{2\a}\right)<\kappa,
\]
then we obtain for $\alpha$ sufficiently close to $1$,
\begin{align*}
&\frac{\log(2\z(2\a)) -\a\log(2\z(2))-\frac{1}{n}\sum_{i=0}^{n-1}\log\left(1
+\left[\frac{|k_{i+1}-k_i|}{k_{i+1}+k_i}\right]^{2\a}\right)}{\frac{n+1}{n}\log(2\z(2))
+\frac{1}{n}\sum_{i=0}^n\log |k_{i+1}-k_i|^{2\a}}
\\&  \geq
-\left(\log(2\z(2\a)) -\a\log(2\z(2))\right) - \frac{\kappa}{\log(2\z(2))
+\frac{1}{n}\sum_{i=0}^n\log |k_{i+1}-k_i|^{2\a}}\\&
\geq -\left(\log(2\z(2\a)) -\a\log(2\z(2))\right) - \kappa \geq - 2\kappa.
\end{align*}
Here we made use of the fact that  
\[
\lim_{\a\to 1} \left( \log(2\z(2\a))  - \a\log(2\z(2))\right)=0,
\]
which implies  $\log(2\z(2\a)) - \a\log(2\z(2)) < \kappa$, for all
$\a$ sufficiently close to $1$.
Also, note that in here the lower bound on $\a$
depends only on $\kappa$ and not 
on $n$. In particular, we also have that $\a$ tends to $1$ as $\kappa$ 
tends to $0$.\\
Before we start with the discussion of the second case,  first note that 
since $0<\left[\frac{|k_{i+1}-k_i|}{k_{i+1}+k_i}\right]^{2\a}\le
1$, we clearly always have  
\begin{equation}\label{9}
\frac{1}{n}\sum_{i=0}^{n-1}\log\left(1+\left[\frac{|k_{i+1}-k_i|}{k_{i+1}+k_i}\right]^{2\a}
\right)\le\log 2.
\end{equation}
Furthermore, since  $k_n$ tends to infinity, there exists 
$j(\kappa)\in \N$ such that
\begin{equation}\label{10}
\log\frac{\kappa^2 k_i^2}{4}>\frac{2 \log 2}{\kappa^2},\quad\text{for all }i\ge 
j(\kappa).
\end{equation}
Let us now come to the second case. That  is, we now assume that for some $n\in \N$ we have 
\[
\frac{1}{n}\sum_{i=0}^{n-1}\log\left(1+\left[\frac{|k_{i+1}-k_i|}{k_{i+1}+
k_i}\right]^{2\a}\right)\ge\kappa.
\] 
Since $x>\log(1+x)$ for all $x>0$, we then  have
\[ \frac{1}{n}\sum_{i=0}^{n-1}\left[\frac{|k_{i+1}-k_i|}{k_{i+1}+k_i}
\right]^{2\a}\ge \kappa.\]
Let us make the following two observations. Firstly,
using the fact that $0<\left[\frac{|k_{i+1}-k_i|}{k_{i+1}+k_i}\right]^{2\a}\le
1$, we can apply Chebyshev's Inequality, which gives that for $n$ suffciently large, 
\[
\card \mathcal{I}_{n} >\kappa n,
\]
where 
\[  \mathcal{I}_{n}:= \left\{i \in [j(\kappa), n]\, 
:\,\left[\frac{|k_{i+1}-k_i|}{k_{i+1}+k_i}\right]^{2\a}\ge\frac{\kappa}{2} 
\right\}.\]
Secondly, note that for $1/2<\a < 1$ the following implication holds. 
\[ \hbox{If} \, \, \,
\left[\frac{|k_{i+1}-k_i|}{k_{i+1}+k_i}\right]^{2\a}\ge\frac{\kappa}{2},\quad 
\, \hbox{then} \, \quad
|k_{i+1}-k_i|>\frac{\kappa}{2} k_i.
\]
Combining these two observations with ~\eqref{10}, we then 
compute
\begin{align*}
\frac{1}{n}\sum_{i=0}^n\log|k_{i+1}-k_i|^{2 \a}
&\ge\frac{\a}{n}\sum_{i\in 
\mathcal{I}_{n}}\log |k_{i+1}-k_i|^{2}
\ge \frac{\a}{n}\sum_{i\in 
\mathcal{I}_{n}}\log\frac{\kappa^2k_i^{2}} {4}\\&
\ge\frac{\a}{n}\sum_{i\in \mathcal{I}_{n}}\frac{2\log 2}{\kappa^2}
> \frac{\log 2}{\kappa^2 n} \card \mathcal{I}_{n}>\frac{\log 2}{ \kappa}.
\end{align*}
Inserting this into our estimate above and using ~\eqref{9}, it follows 
\begin{align*}
&\frac{\log(2\z(2\a)) -\a\log(2\z(2))-\frac{1}{n}\sum_{i=0}^{n-1}\log\left(1
+\left[\frac{|k_{i+1}-k_i|}{k_{i+1}+k_i}\right]^{2\a}\right)}{\frac{n+1}{n}\log(2\z(2))
+\frac{1}{n}\sum_{i=0}^n\log |k_{i+1}-k_i|^{2\a}}
\\&  \geq \frac{-\log 2}{\log(2\z(2)) + \frac{\log 2}{ \kappa}} = 
-\kappa \frac{\log 2}{\kappa \log(2\z(2)) + \log 2} \\&
\geq - \kappa.
\end{align*}
This finishes the second case.\\
Combining the latter results with  ~\eqref{2nd}, and putting $\e :=3 
\kappa$, we 
have now shown that
\[
\liminf_{n\to\infty}\frac{\log\frac{\mu_\a(I_{k_1\cdots
    k_n})}{|I_{k_1\cdots k_n}|^\a}}{\log|I_{k_1\cdots k_n k_{n+1}}|}\ge 
     -\e,
\]
and hence,
\[  \liminf_{n \to \infty} \frac{\log\mu_\a(I_{k_1\cdots k_n})}{\log|I_{k_1\cdots
    k_n}|}  \geq \a -  \e,\]
Since $\a$ tends to $1$ for $\e$ 
tending to $0$,  the proof is complete

\end{proof}

\subsection{The proof of the Main Theorem}

$\,$

\noindent Recall that the lower pointwise dimension $\nd
(x)$ of a Borel measure $\nu$ on $\R$ at a point $x \in \R$ is given by
\[ \nd(x) := \liminf_{r\to 
0}\frac{\log\nu (B(x,r))}{\log r},\]
where $B(x,r)$ refers to the interval centred at $x$ with diameter 
equal to $2r$.
The idea is to apply the well-known
Mass Distribution Principle of Frostman \cite{Fr} and Billingsley \cite{Bi}
(see also e.g. \cite{F}).

\noindent In order to be able to apply the Mass Distribution Principle, 
we still require 
   the following straight forward generalization  of  Furstenberg's 
   Lemma  \cite{Fu}.
   
   \begin{lemma}\label{furst}
Let $\nu$ be a Borel measure on $\R$, and let  
$(r_{n})$ be a sequence of positive numbers for 
which $\lim_{n \to \infty} r_{n}=0$ and $\lim_{n\to \infty} ( \log 
r_{n+1} / \log r_{n})=1$. We then have for every $x \in \R$,
   \[
  \nd (x)=\liminf_{n\to\infty}\frac{\log\nu(B(x,r_n))}{\log r_n}.
   \]
   \end{lemma}
   \begin{proof}
  For  $r>0$ we define $n=n(r) :=\max\{k\in\N\, :\, r_k \ge r\}$. 
  The assertion of the lemma is then an immediate consequence
  of the following simple calculation.
 \[
   \frac{\log\nu(B(x,r))}{\log r} \ge\frac{\log
   \nu(B(x,r_n))}{\log r_{n+1}}
   = \frac{\log r_{n}}{\log r_{n+1}}  \frac{\log
   \nu(B(x,r_n))}{\log r_{n}}.
 \]
   \end{proof}

    \vspace{2mm}
    
\noindent
{\em Proof of the Main Theorem.}
Let $\e>0$ be given, and then fix $1/2<\a <1$ and $\g>1/(2\a-1)$
as in Proposition~\ref{lpd}.  
By Lemma~\ref{slow} we 
have that in order to find a lower bound for $\dim_{H}(\mu_{\a})$ 
it is sufficient to give an estimate for $\ud(x)$ from below, for
each $x = \r (k_{1},k_{2},\ldots ) \in \c_\infty^\g$.   For this note that 
 Proposition~\ref{lpd} implies 
\[
\liminf_{n\to\infty}\frac{\log\mu_\a(I_{k_1\cdots k_n})}{\log |I_{k_1\cdots k_n}|}\ge \a-\e.
\]
In order to deduce the desired lower bound for $\ud(x)$,  we then use
~\eqref{1} and the definition of $\c_\infty^\g$, which gives for 
$r_n:=|I_{k_1\cdots k_n}|$,
\begin{align*}
    \lim_{n\to \infty} \frac{\log 
    r_{n+1}}{ \log r_{n}}&=   1+ \lim_{n\to \infty}
\frac{\log\frac{|I_{k_1\cdots
      k_nk_{n+1}}|}{|I_{k_1\cdots k_n}|}}{\log|I_{k_1\cdots 
      k_n}|}\\& \le
1+  \lim_{n\to \infty} \frac{\g\log n}{n\log(2\z(2))}  =1.
\end{align*}
Therefore, Lemma~\ref{furst} implies that for each 
$x\in \c_\infty^\g$, 
\[
\ud(x)\ge\a-\e.
\]
Combining this with Corollary~\ref{C},
Corollary~\ref{transience} and Lemma~\ref{slow}, we have by the Mass Distribution 
Principle, 
\[
\dim_H \left( \c \right) \geq \dim_H \left( \c_\infty \right) \geq \dim_H \left( \c_\infty^\g 
\right)\ge\dim_H \left(\mu_\a\right) \ge\a-\e.
\]
Finally, note that by Proposition~\ref{lpd} we have that $\a$ tends to 
$1$ for  $\e$ tending to $0$. 
This completes the proof of the theorem.

$\,$ \hfill  $\Box$

\end{document}